\documentclass[conference,twoside,10pt]{ieeeconf}
\usepackage{cite}
\usepackage{amsmath,amssymb,amsfonts}
\usepackage{graphicx}
\usepackage{subcaption}
\usepackage{booktabs}
\usepackage{empheq}
\usepackage{xcolor}
\newcommand{\R}{\mathbb{R}}
\newtheorem{theorem}{Theorem}
\newtheorem{proposition}{Proposition}

\newtheorem{lemma}{Lemma}
\newtheorem{definition}{Definition}
\newtheorem{remark}{Remark}
\newtheorem{assumption}{Assumption}
\newtheorem{problem}{Problem}
\usepackage{hyperref}
\makeatletter

\let\ACC@bibitem\bibitem
\renewcommand\bibitem[1]{\ACC@bibitem{#1}%
  \Hy@raisedlink{\hyper@anchorstart{bib.\the\c@enumiv}\hyper@anchorend}}
\makeatother
\title{\LARGE\bf How Much Spatial Control Is Enough? Subdomain Optimal Control of Reaction--Diffusion Systems in Synthetic Developmental Biology}

\author{
  Mohamed Amine Ouchdiri, Hamza Faquir, Saad Benjelloun,
  Adnane Saoud
\thanks{Mohamed Amine Ouchdiri and Adnane Saoud are with the College of
Computing, University Mohammed VI Polytechnic, Benguerir, Morocco (e-mail: amine.ouchdiri@um6p.ma; adnane.saoud@um6p.ma).}
  \thanks{Hamza Faquir is with the Department of Bioproducts and Biosystems, Aalto University, Espoo, Finland (e-mail: hamza.faquir@aalto.fi).}
  \thanks{Saad Benjelloun is with De Vinci Research Center, Paris, France (e-mail: saad.benjelloun@devinci.fr).}
}

\begin{document}
\maketitle
\thispagestyle{empty}\pagestyle{empty}

\begin{abstract}
Reaction--diffusion systems can produce spatial patterns such as stripes and spots by diffusion-driven instability. In synthetic developmental biology, such patterns are engineered in cell populations, and steering them toward a prescribed target pattern can be formulated as an optimal control problem. When the control acts on the entire spatial domain, the existence of optimal controls and the first-order optimality conditions are well established. In experiments, the control is applied with light by optogenetic tools, and the light illuminates a chosen region of the tissue, so the control acts on a subdomain.
We consider the setting where the control is restricted to a subdomain fixed in advance, for a synthetic circuit in which the slowly diffusing activator Nodal and the fast-diffusing inhibitor Lefty form the pattern. We study the optimality loss, the difference between the optimal costs of the subdomain problem and of the full-domain problem, and we ask how much coverage keeps it within a desired precision. In simulations with one parameter set and two target patterns, nine squares give a smaller terminal error than one square at the same coverage. Moreover, the error depends on the average distance from the uncontrolled points to the control region. In particular, small errors occur when this distance is comparable to the length scale set by Nodal diffusion and degradation.
\end{abstract}

\section{Introduction}\label{sec1}
Reaction--diffusion systems can develop stable spatial patterns~\cite{Turing1952}. In such systems, two species with different diffusion rates interact and produce structures such as stripes and spots on two-dimensional domains. The problem of steering such a system toward a prescribed target pattern can be formulated as an optimal control problem for coupled reaction--diffusion equations, within the classical theory of optimal control of partial differential equations~\cite{Lions1971,Troeltzsch2010,LasieckaTriggiani2000}. Optimal controls were computed for fronts and spirals that propagate in bistable and excitable media~\cite{CasasRyllTroltzsch2013}. For stationary Turing patterns, which form by diffusion-driven instability in activator--inhibitor systems such as the one considered here, matching a target pattern was posed as an identification problem~\cite{GarvieTrenchea2014}. In these works, the control acts on the whole spatial domain.

In many applications, the control acts on a part of the spatial domain, and several approaches address this restriction. For controllability, the heat equation can be steered to zero from any open subdomain, in arbitrarily small time~\cite{LebeauRobbiano1995}. The control that does this becomes arbitrarily large as the time tends to zero, at a rate bounded by the largest distance from a point of the domain to the chosen subdomain~\cite{Miller2004}. When the control is required to be nonnegative, steering to zero becomes possible once the time exceeds a positive minimal value~\cite{Loheac2017}. For stabilization, feedback laws acting on a fixed subdomain make an equilibrium of a parabolic equation exponentially stable~\cite{LasieckaTriggiani2000}. The same objective is reached from the boundary of the domain by backstepping, which constructs stabilizing feedback laws for reaction--diffusion equations~\cite{KrsticSmyshlyaev2008}, and in both cases the region where the control acts is given in advance. Choosing this region is the actuator placement problem, in which the location of an actuator of given size is a design variable, optimized to minimize a criterion of the controlled system such as its linear-quadratic cost~\cite{Kubrusly1985,Morris2011}, and, for semilinear systems, optimized jointly with the control~\cite{Edalatzadeh2019}. In optimal control, the control acts on a subdomain chosen beforehand in~\cite{Buchholz2013}. In sparse optimal control, an $L^1$ cost concentrates the control on a small part of the domain, which the optimization selects~\cite{Stadler2009}. In the time-dependent case, an $L^1$ cost in space and an $L^2$ cost in time keep that part active over the whole time interval~\cite{Herzog2012,Casas2017}. In this paper the subdomain is prescribed, the controls are bounded and nonnegative, and the horizon is fixed. We quantify how the size and the arrangement of the subdomain affect the optimality loss and the accuracy of reaching a target pattern.

In synthetic developmental biology, the Nodal--Lefty two-species system~\cite{Ouchdiri2025_MathBio} was created in synthetic cell populations~\cite{Sekine2018} and optogenetic tools control Nodal signaling with light applied to a chosen region, as shown in zebrafish embryos~\cite{McNamara2025} and reviewed in~\cite{Krueger2019}. In~\cite{CDC_paper}, the problem of steering this system to a prescribed target was studied with the control acting on the entire spatial domain, and the existence of optimal controls and first-order optimality conditions were established. Restricting the control to a subdomain can cause a loss of optimality, which raises two questions. Which coverage of the domain guarantees a prescribed accuracy on the target pattern? For a fixed total area, does one region or several regions give better accuracy? The answers can guide the choice of the size and arrangement of the controlled region in experiments~\cite{Ebrahimkhani2019}.

In this paper, we address these questions for the Nodal--Lefty system by taking an optimal control of the full-domain problem and setting it to zero outside the subdomain. This truncated control satisfies the input constraints and is supported on the subdomain, so it is admissible for the subdomain problem, and its cost is an upper bound on the optimal cost of the subdomain problem. Using the adjoint state of the full-domain problem, we derive an identity for the difference between the costs of the two controls. Bounding its terms gives an upper bound on the optimality loss proportional to the squared space--time $L^2$ norm of the optimal control outside the subdomain, and a coverage that keeps the optimality loss within a desired precision. In the simulations, we use one parameter set and two target patterns to compare full-domain and subdomain control. For the subdomain cases, we consider one square and nine squares at five coverage levels, comparing controls optimized for each region with the truncated full-domain control. We then examine how the terminal error varies with the mean distance from the uncontrolled points to the control region. In particular, we observe small terminal errors when this mean distance is comparable to the characteristic length set by Nodal diffusion and degradation.

The remainder of this paper is organized as follows. Section~\ref{sec2} formulates the problem. Section~\ref{sec3} establishes well-posedness and the existence of optimal controls. Section~\ref{sec4} presents the main results. Section~\ref{sec5} presents numerical results.

\textit{Notations.} For a bounded domain $\Omega\subset\R^2$, we denote by $L^2(\Omega)$ and $L^\infty(\Omega)$ the usual Lebesgue spaces, with norms $\|\cdot\|_{L^2(\Omega)}$ and $\|\cdot\|_{L^\infty(\Omega)}$, and by $L^2(Q)$ and $L^\infty(Q)$ the usual Lebesgue spaces on the space--time cylinder $Q:=\Omega\times(0,T)$. For $f=(f_1,f_2)$, we let $\|f\|_{(L^2(\Omega))^2}:=(\|f_1\|^2_{L^2(\Omega)}+\|f_2\|^2_{L^2(\Omega)})^{1/2}$ and $\|f\|_{(L^2(Q))^2}:=(\|f_1\|^2_{L^2(Q)}+\|f_2\|^2_{L^2(Q)})^{1/2}$, and we write $\|f\|$ for $\|f\|_{(L^2(\Omega))^2}$. For $k\in\{1,2\}$, we denote by $H^k(\Omega)$ the space of functions whose derivatives up to order $k$ belong to $L^2(\Omega)$. For a Banach space $V$, $L^2(0,T;V)$ and $C([0,T];V)$ denote the spaces of square-integrable and of continuous functions from $[0,T]$ to $V$, and $H^1(0,T;V)$ the space of functions in $L^2(0,T;V)$ whose time derivative belongs to $L^2(0,T;V)$. We denote by $|\cdot|$ the absolute value of a real number, the Euclidean norm in $\R^2$, the associated norm of $2\times 2$ matrices, and the Lebesgue measure of a set, by $\|\cdot\|_\infty$ the maximum norm in $\R^2$, and we let $\R_+:=[0,\infty)$ and $s^+:=\max\{s,0\}$ for $s\in\R$. For $a<b$, we let $\mathrm{Proj}_{[a,b]}(s):=\max\{a,\min\{s,b\}\}$, and for $x\in\R^2$ and a nonempty $A\subseteq\R^2$, $\mathrm{dist}(x,A):=\inf_{y\in A}|x-y|$.

\section{Problem Formulation}\label{sec2}
In this section, we introduce the state equation, the cost functional, the standing assumption, and the two control configurations. We work on the square spatial domain $\Omega=(0,L)^2\subset\R^2$ over a fixed time horizon $T>0$. The space--time cylinder is $Q:=\Omega\times(0,T)$, and $\Sigma:=\partial\Omega\times(0,T)$ is its lateral boundary.
\subsection{State equation}
We consider the Nodal--Lefty reaction--diffusion system studied in~\cite{Ouchdiri2025_MathBio}.
Let $\omega\subseteq\Omega$ be a measurable subset representing the control region. The Nodal and Lefty concentrations, denoted $z_n$ and $z_l$, evolve according to
\begin{subequations}\label{eq1}
\begin{empheq}[left=\empheqlbrace]{align}
&\partial_t z_n = D_n\Delta z_n-\gamma_n z_n+\alpha_n H(z_n,z_l)+\beta_n\chi_\omega u_n, \\
&\partial_t z_l = D_l\Delta z_l-\gamma_l z_l+\alpha_l H(z_n,z_l)+\beta_l\chi_\omega u_l, \\
&\partial_\nu z_n\big|_\Sigma = \partial_\nu z_l\big|_\Sigma = 0, \\
&z_n(\cdot,0) = z^0_n,\quad z_l(\cdot,0) = z^0_l,
\end{empheq}
\end{subequations}
where $z_n,z_l\colon\bar\Omega\times[0,T]\to\R_+$, $(x,t)\mapsto z_i(x,t)$, are the state variables, $u=(u_n,u_l)\colon Q\to\R_+^2$ is the control input, $\Delta z_i:=\partial^2 z_i/\partial x_1^2+\partial^2 z_i/\partial x_2^2$, $i\in\{n,l\}$, is the Laplacian with respect to $x=(x_1,x_2)\in\Omega$, and $\partial_\nu$ is the outward normal derivative on $\partial\Omega$. The constants $D_n,D_l>0$ are the diffusion coefficients, $\gamma_n,\gamma_l>0$ are the degradation rates, $\alpha_n,\alpha_l>0$ are the production rates, and $\beta_n,\beta_l>0$ are the input gains. The initial conditions $z^0_n,z^0_l$ satisfy Assumption~\ref{As1} below. The homogeneous Neumann conditions model a domain with no flux across $\partial\Omega$. The control acts on the state by the terms $\beta_i\chi_\omega u_i$, where $\chi_\omega\colon\Omega\to\{0,1\}$ is the characteristic function of $\omega$, defined by $\chi_\omega(x)=1$ if $x\in\omega$ and $\chi_\omega(x)=0$ otherwise.
The function $H$ is given by
\begin{equation}\label{eq2}
H(z_n,z_l)
:= \frac{z_n^{n_n}}
{z_n^{n_n}
+ \bigl[k_n\bigl(1
+ (z_l/k_l)^{n_l}\bigr)
\bigr]^{n_n}},
\end{equation}
which takes values in $[0,1]$, increases with $z_n$, and decreases with $z_l$, so Nodal activates the production of both species and Lefty inhibits it. The Hill coefficients $n_n,n_l\geq 1$ and the half-saturation constants $k_n,k_l>0$ are fixed parameters. Since $n_n,n_l\geq 1$ and $k_n,k_l>0$, the partial derivatives of $H$ are bounded on $(0,\infty)^2$, so $H$ is globally Lipschitz on $\R_+^2$, that is, there exists $L_H>0$ such that
\begin{equation}\label{eq3}
|H(a)-H(b)|\leq L_H\,|a-b|
\quad\text{for all } a,b\in\R_+^2.
\end{equation}
Moreover, $H$ is $C^\infty$ on $(0,\infty)^2$, since the powers in~\eqref{eq2} are smooth for $z_n,z_l>0$ and the denominator is strictly positive.

\subsection{Cost functional and admissible controls}

The goal is to steer the system from an initial Turing pattern to a prescribed target $z_d=(z_{n,d},z_{l,d})$ at the final time~$T>0$. The cost functional is
\begin{align}\label{eq4}
J(u;\omega)
&:= \frac{\mu}{2}\int_\Omega
  \bigl|z_n(x,T)-z_{n,d}(x)\bigr|^2\,dx
\notag\\
&\quad + \frac{\mu}{2}\int_\Omega
  \bigl|z_l(x,T)-z_{l,d}(x)\bigr|^2\,dx
\notag\\
&\quad + \frac{\lambda}{2}
  \int_0^T\!\!\int_\omega
  \bigl(u_n^2+u_l^2\bigr)\,dx\,dt,
\end{align}
where $z=(z_n,z_l)$ is the solution of~\eqref{eq1} corresponding to~$u$, and the weights $\mu>0$ and $\lambda>0$ balance the terminal error and the control cost. The first two terms penalize the mismatch with the target over the whole of $\Omega$, and the third penalizes the control effort on $\omega$. The set of admissible controls is
\begin{align}\label{eq5}
U_{\mathrm{ad}}
:= \bigl\{u\in (L^\infty(Q))^2 :\;
  &0\leq u_i\leq u_{b,i}
\notag\\
  &\text{a.e.\ in }Q,\;
  i\in\{n,l\}
\bigr\},
\end{align}
where $u_{b,n},u_{b,l}>0$ are the fixed upper bounds.
\begin{assumption}\label{As1}
The initial conditions and the target satisfy
\begin{enumerate}
\item[\emph{(i)}] $z^0_n,z^0_l\in H^1(\Omega)\cap L^\infty(\Omega)$ and $z^0_n,z^0_l\geq m$~a.e.\ in~$\Omega$, for some $m>0$,
\item[\emph{(ii)}] $z_d\in(H^1(\Omega)\cap L^\infty(\Omega))^2$.
\end{enumerate}
\end{assumption}
\begin{remark}
Item~(i) gives the regularity and the bounds of the state in Lemma~\ref{Lem1}, and the lower bound $m$ keeps the states in a region where the second derivatives of $H$ are bounded. Item~(ii) gives the regularity and the boundedness of the adjoint state in Proposition~\ref{Prop2}(ii), which are needed for the pointwise estimates in Section~\ref{sec4}. Both items are natural and met in practice: in the biological setting the initial concentrations are positive and bounded, and the target is a pattern produced by~\eqref{eq1}.
\end{remark}
\subsection{Control configurations and optimality loss}

We now define the two control configurations studied in this paper. When $\omega=\Omega$ the control acts on the entire domain and the optimal cost is
\begin{equation}\label{eq6}
J_\Omega^\star
:= \min_{u\in U_{\mathrm{ad}}} J(u;\Omega).
\end{equation}
For a subdomain $\omega_\varepsilon\subseteq\Omega$, we consider two cases.

In the first case, $\omega_\varepsilon$ is a single square of half-side $\varepsilon\in(0,L/2]$, where $L$ is the side length of $\Omega$, centered at $x_0\in\Omega$,
\begin{equation}\label{eq7}
\omega_\varepsilon
:= \bigl\{\xi\in\Omega :
  \|\xi-x_0\|_\infty<\varepsilon\bigr\},
\end{equation}
where $x_0$ lies at distance at least $\varepsilon$ from $\partial\Omega$ in the norm $\|\cdot\|_\infty$, so that its area is $|\omega_\varepsilon|=4\varepsilon^2$.

In the second case, $\omega_\varepsilon$ is a union of $M$ disjoint squares,
\begin{equation}\label{eq8}
\omega_\varepsilon
:= \bigcup_{k=1}^M \omega_k,
\end{equation}
where $\omega_k$ is the square~\eqref{eq7} with center $x^{(k)}\in\Omega$ and half-side $\varepsilon_k>0$, and $\sum_{k=1}^M 4\varepsilon_k^2 = 4\varepsilon^2$, so that the total control area is the same $|\omega_\varepsilon|=4\varepsilon^2$.
The optimal cost of the subdomain problem is
\begin{equation}\label{eq9}
J_{\omega}^\star
:= \min_{u\in U_{\mathrm{ad}}}
J(u;\omega).
\end{equation}
Under Assumption~\ref{As1}, problems~\eqref{eq6} and~\eqref{eq9} admit a minimizer, by Proposition~\ref{Prop2}(i) in Section~\ref{sec3}. Moreover, since the state of~\eqref{eq1} with control region $\omega$ and control $u$ equals the state with control region $\Omega$ and control $\chi_\omega u$, we have $J(u;\omega)=J(\chi_\omega u;\Omega)$ for every $u\in U_{\mathrm{ad}}$. The \emph{optimality loss} of the control region $\omega$ is
\begin{equation}\label{eq10}
\delta(\omega)
:= J_{\omega}^\star
- J_\Omega^\star,
\end{equation}
which measures how much the optimal cost rises when the control is restricted to $\omega$. Indeed, for every $u\in U_{\mathrm{ad}}$ we have $\chi_\omega u\in U_{\mathrm{ad}}$, so $J_\Omega^\star\leq J(u;\omega)$. Consequently, taking the minimum over $u$ gives $J_\Omega^\star\leq J_\omega^\star$, and therefore $\delta(\omega)\geq 0$, with $\delta(\Omega)=0$. In both configurations, the centers $x_0$ and $x^{(1)},\ldots,x^{(M)}$ are fixed in advance, and the half-side $\varepsilon$ determines the coverage of the subdomain $|\omega_\varepsilon|/|\Omega|$.

With these definitions, the two questions of Section~\ref{sec1} can be stated in terms of the optimality loss as the following problem.

\medskip
\noindent\colorbox{gray!15}{\parbox{0.93\columnwidth}{
\begin{problem}
(Q1) Given a desired precision $\eta>0$, which coverage $|\omega_\varepsilon|/|\Omega|$ is sufficient to guarantee $\delta(\omega_\varepsilon)\leq\eta$?

\smallskip
(Q2) For a fixed total control area, how does the placement of the control, on a single subdomain or over several subdomains, affect the optimality loss?
\end{problem}
}}

\section{Well-Posedness and Optimality Conditions}\label{sec3}
The main results in Section~\ref{sec4} rely on the well-posedness of the state equation, the Lipschitz continuity of the control-to-state operator, and the first-order optimality conditions, which we collect here. These results are stated in~\cite{CDC_paper} for the control acting on $\Omega$; here they hold for an arbitrary measurable $\omega\subseteq\Omega$, with a Lipschitz constant that does not depend on $\omega$.
\subsection{Well-posedness}
\begin{lemma}\label{Lem1}
Let Assumption~\ref{As1} hold, let $\omega\subseteq\Omega$ be measurable, and let $u\in U_{\mathrm{ad}}$. Then the state equation~\eqref{eq1}, with control region $\omega$ and control $u$, admits a unique strong solution $z=(z_n,z_l)$ such that, for each $i\in\{n,l\}$,
\begin{equation}\label{eq11}
z_i\in H^1(0,T;L^2(\Omega))
\cap L^2(0,T;H^2(\Omega))
\cap L^\infty(Q)
\end{equation}
and $m_i\leq z_i\leq M_i$ a.e.\ in $Q$, where $m$ is the lower bound in Assumption~\ref{As1}(i) and
\begin{subequations}\label{eq12}
\begin{align}
m_i&:=m\,e^{-\gamma_i T},\\
M_i&:=\max\Big\{\|z^0_i\|_{L^\infty(\Omega)},\frac{\alpha_i+\beta_i u_{b,i}}{\gamma_i}\Big\}.
\end{align}
\end{subequations}
Moreover, the Hessian matrix $D^2H$ of $H$ satisfies $\|D^2H\|_{L^\infty(K_m)}:=\sup_{z\in K_m}|D^2H(z)|<\infty$, where $K_m:=[m_n,M_n]\times[m_l,M_l]$.
\end{lemma}
\begin{proof}
We prove in turn the existence and uniqueness of the strong solution, the bounds on $z$, and the bound on $D^2H$. Since $H$ is globally Lipschitz on $\R_+^2$ by~\eqref{eq3}, the first follows from a fixed-point argument; see~\cite[p.~175]{Barbu2012}. For the second, the method of upper and lower solutions~\cite[p.~417]{Pao1992}, with the lower solution $(m e^{-\gamma_n t},m e^{-\gamma_l t})$ and the upper solution $(M_n,M_l)$, gives $m e^{-\gamma_i t}\leq z_i\leq M_i$ a.e.\ in $Q$, since $0\leq H\leq 1$ and $0\leq u_i\leq u_{b,i}$. It remains to note that $K_m$ is a compact subset of $(0,\infty)^2$, on which $H$ is $C^\infty$, so $D^2H$ is bounded on $K_m$.
\end{proof}

\subsection{Lipschitz continuity of the control-to-state operator}

In order to write the state as a function of the control, we fix the horizon $T$, the domain $\Omega$, the initial condition $z^0$ and the parameters of~\eqref{eq1}. Let the control region $\omega\subseteq\Omega$ be measurable. By Lemma~\ref{Lem1} and the embedding $H^1(0,T;L^2(\Omega))\subset C([0,T];L^2(\Omega))$~\cite{Evans2010}, each control $u\in U_{\mathrm{ad}}$ determines a unique strong solution $z$ of~\eqref{eq1}, with $z_i\in C([0,T];L^2(\Omega))$ for $i\in\{n,l\}$. Hence we define the control-to-state operator
\begin{equation}\label{eq14}
S_\omega\colon U_{\mathrm{ad}}\to C([0,T];(L^2(\Omega))^2),\quad S_\omega(u):=z.
\end{equation}
The next result shows that $S_\omega$ is Lipschitz.
\begin{proposition}\label{Prop1}
Let Assumption~\ref{As1} hold, let the control region $\omega\subseteq\Omega$ be measurable, and let $S_\omega$ be the control-to-state operator~\eqref{eq14}, which maps each control $u\in U_{\mathrm{ad}}$ to the strong solution of~\eqref{eq1}. Then $S_\omega$ is Lipschitz continuous on $U_{\mathrm{ad}}$ for the $(L^2(Q))^2$ norm. For any $\sigma>0$, define
\begin{equation}\label{eq15}
L_S := \frac{\beta_{\max}}{\sqrt{\sigma}}\,
e^{(\sqrt{2}\,\alpha_{\max}L_H+\sigma/2)T},
\end{equation}
where $\alpha_{\max}:=\max\{\alpha_n,\alpha_l\}$ and $\beta_{\max}:=\max\{\beta_n,\beta_l\}$ are taken from~\eqref{eq1}, and $L_H$ is the Lipschitz constant of $H$ in~\eqref{eq3}. Then, for all $u^1,u^2\in U_{\mathrm{ad}}$,
\begin{align}\label{eq16}
\sup_{t\in[0,T]}
&\|S_\omega(u^1)(t)
  -S_\omega(u^2)(t)\|_{(L^2(\Omega))^2}
\notag\\
&\qquad\leq L_S\,
  \|u^1-u^2\|_{(L^2(Q))^2}.
\end{align}
The constant $L_S$ depends only on $T$ and on the parameters of~\eqref{eq1}, hence not on $\omega$.
\end{proposition}

The proof is given in Appendix~\ref{app1}.

\subsection{Existence and characterization of optimal controls}

The next result gives the existence of a minimizer for every control region, together with the first-order conditions satisfied by a minimizer of the full-domain problem.
\begin{proposition}\label{Prop2}
Let Assumption~\ref{As1} hold.
\begin{enumerate}
\item[\emph{(i)}] For every measurable $\omega\subseteq\Omega$, the restricted problem~\eqref{eq9} admits at least one minimizer.
\item[\emph{(ii)}] Let $u^*_\Omega\in U_{\mathrm{ad}}$ be a minimizer of the full-domain problem~\eqref{eq6} and let $z^*:=S_\Omega(u^*_\Omega)$ be the corresponding state of~\eqref{eq1}. There exists a unique adjoint state $p^*=(p^*_n,p^*_l)$, with $p^*_i\in H^1(0,T;L^2(\Omega))\cap L^2(0,T;H^2(\Omega))\cap L^\infty(Q)$, satisfying, backward in $Q$,
\begin{equation}\label{eq17}
\left\{
\begin{aligned}
-\partial_t p^*_n &= D_n\Delta p^*_n
  - \gamma_n p^*_n + \alpha_n\partial_{z_n}H(z^*)\,p^*_n\\
  &+ \alpha_l\partial_{z_n}H(z^*)\,p^*_l,\\
-\partial_t p^*_l &= D_l\Delta p^*_l
  - \gamma_l p^*_l+ \alpha_n\partial_{z_l}H(z^*)\,p^*_n\\
  &+ \alpha_l\partial_{z_l}H(z^*)\,p^*_l,\\
\partial_\nu p^*_n\big|_\Sigma
  &= \partial_\nu p^*_l\big|_\Sigma = 0,\\
p^*_i(\cdot,T)
  &= \mu(z^*_i(\cdot,T)-z_{i,d}),
  \quad i\in\{n,l\}.
\end{aligned}
\right.
\end{equation}
The control $u^*_\Omega$ satisfies the projection formula
\begin{equation}\label{eq18}
u^*_{i,\Omega}
= \mathrm{Proj}_{[0,\,u_{b,i}]}
\!\Big(\!-\frac{\beta_i}{\lambda}\,
p^*_i\Big)
\quad\text{a.e.\ in }Q,
\end{equation}
for $i\in\{n,l\}$, where $\mu$ and $\lambda$ are the weights in~\eqref{eq4} and $0$ and $u_{b,i}$ are the bounds in~\eqref{eq5}.
\item[\emph{(iii)}] The sign condition
\begin{equation}\label{eq19}
u^*_{i,\Omega}\,p^*_i \leq 0
\quad\text{a.e.\ in }Q
\end{equation}
holds for each $i\in\{n,l\}$.
\end{enumerate}
\end{proposition}
\begin{proof}
We prove the three items in turn. For item~(i), we reduce~\eqref{eq9} to a full-domain problem whose control bounds are supported on $\omega$. Since these bounds belong to $L^\infty(Q)$, the cost $J(\cdot\,;\Omega)$ admits a minimizer over the controls $v$ with $0\leq v_i\leq\chi_\omega u_{b,i}$ by~\cite{CDC_paper}, and the identity $J(u;\omega)=J(\chi_\omega u;\Omega)$ shows that this minimizer also minimizes $J(\cdot\,;\omega)$ over $U_{\mathrm{ad}}$. Item~(ii) is established in~\cite{CDC_paper}, and~\eqref{eq18} follows pointwise from the variational inequality there. The bound $p^*\in(L^\infty(Q))^2$ then follows from~\eqref{eq17}: the terminal condition gives $p^*_i(\cdot,T)\in L^\infty(\Omega)$ by Lemma~\ref{Lem1} and Assumption~\ref{As1}(ii), and the coefficients $\partial_{z_j}H(z^*)$ are bounded by $L_H$; hence testing~\eqref{eq17} with $(p^*_i-\theta)^+$ and $(-p^*_i-\theta)^+$, where $\theta(t):=\theta_0e^{2\alpha_{\max}L_H(T-t)}$ and $\theta_0:=\max_i\|p^*_i(\cdot,T)\|_{L^\infty(\Omega)}$, gives $|p^*_i|\leq\theta$ a.e.\ in $Q$ by Gr\"onwall's inequality. Finally, item~(iii) follows from~\eqref{eq18}. Where $p^*_i\leq 0$, the product $u^*_{i,\Omega}p^*_i$ is nonpositive because $u^*_{i,\Omega}\geq 0$. Where $p^*_i>0$, we have $-\beta_ip^*_i/\lambda<0$, so~\eqref{eq18} projects a negative quantity onto $[0,u_{b,i}]$ and gives $u^*_{i,\Omega}=0$, which makes the product equal to zero.
\end{proof}

Since the state equation~\eqref{eq1} is nonlinear, the cost $J(\cdot\,;\Omega)$ is in general nonconvex. Conditions~\eqref{eq17} and~\eqref{eq18} are therefore necessary and not sufficient: a control that satisfies them need not be a global minimizer of~\eqref{eq6}.
\begin{definition}\label{RemLoc}
 Let Assumption~\ref{As1} hold. A pair $(u,p)$ with $u\in U_{\mathrm{ad}}$ satisfies the \emph{optimality conditions} when $p$ solves the adjoint system~\eqref{eq17} for the state $S_\Omega(u)$, and $u$ and $p$ satisfy the projection formula~\eqref{eq18}.
\end{definition}

By Proposition~\ref{Prop2}, a minimizer of~\eqref{eq6} and its adjoint satisfy the optimality conditions. Accordingly, Proposition~\ref{Prop3} is stated for every pair of Definition~\ref{RemLoc}, whereas Theorem~\ref{Th1} uses $J(u^*_\Omega;\Omega)=J_\Omega^\star$ and therefore assumes that $u^*_\Omega$ is a minimizer.

\section{Optimality Loss Analysis}\label{sec4}

In this section, we first show that the optimality loss $\delta(\omega_\varepsilon)$ defined in~\eqref{eq10} is non-increasing as the control region grows. We then bound it and derive a coverage that keeps it within a desired precision.

\subsection{Monotonicity of the optimal cost}
The next lemma shows that the optimality loss is non-increasing in the control region.
\begin{lemma}\label{Lem2}
Let Assumption~\ref{As1} hold.
\begin{enumerate}
\item[\emph{(i)}] Let $\omega_1\subseteq\omega_2\subseteq\Omega$ be measurable. Then the optimal costs~\eqref{eq9} of the two control regions satisfy $J_{\omega_2}^\star\leq J_{\omega_1}^\star$, and the optimality losses~\eqref{eq10} satisfy $\delta(\omega_2)\leq\delta(\omega_1)$.
\item[\emph{(ii)}] Let $I\subseteq(0,L/2]$ be an interval, let $(\omega_\varepsilon)_{\varepsilon\in I}$ be a family of measurable subsets of $\Omega$ that increases with $\varepsilon$, that is, $\omega_\varepsilon\subseteq\omega_{\varepsilon'}$ whenever $\varepsilon\leq\varepsilon'$, let $\eta>0$, and let $\varepsilon^\star_\eta$ be the infimum of the values of $\varepsilon$ that keep the optimality loss within $\eta$, $\varepsilon^\star_\eta:=\inf\{\varepsilon\in I:\delta(\omega_\varepsilon)\leq\eta\}$, with $\inf\emptyset:=+\infty$. Then $\delta(\omega_\varepsilon)\leq\eta$ for every $\varepsilon\in I$ with $\varepsilon>\varepsilon^\star_\eta$.
\end{enumerate}
\end{lemma}
\begin{proof}
The proof rests on the identity $J(u;\omega)=J(\chi_\omega u;\Omega)$. (i) Let $u\in U_{\mathrm{ad}}$ and $v:=\chi_{\omega_1}u\in U_{\mathrm{ad}}$. Since $\chi_{\omega_2}\chi_{\omega_1}=\chi_{\omega_1}$, the identity gives $J(v;\omega_2)=J(\chi_{\omega_1}u;\Omega)=J(u;\omega_1)$. Hence $J_{\omega_2}^\star\leq J(u;\omega_1)$ for every $u\in U_{\mathrm{ad}}$, so $J_{\omega_2}^\star\leq J_{\omega_1}^\star$, and subtracting $J_\Omega^\star$ gives $\delta(\omega_2)\leq\delta(\omega_1)$. (ii) Let $\varepsilon>\varepsilon^\star_\eta$. The definition of the infimum gives $\varepsilon'\in I$ with $\varepsilon'<\varepsilon$ and $\delta(\omega_{\varepsilon'})\leq\eta$, and (i) with $\omega_{\varepsilon'}\subseteq\omega_\varepsilon$ gives $\delta(\omega_\varepsilon)\leq\delta(\omega_{\varepsilon'})\leq\eta$.
\end{proof}

\subsection{Optimality loss bound and sufficient coverage}

In this subsection, we fix a control $u^*_\Omega\in U_{\mathrm{ad}}$ acting on the whole domain and an adjoint state $p^*$ such that the pair $(u^*_\Omega,p^*)$ satisfies the optimality conditions of Definition~\ref{RemLoc}. Moreover, we denote by $z^*:=S_\Omega(u^*_\Omega)$ the state of~\eqref{eq1} under this control, with $\omega=\Omega$. The terminal error under full-domain control is
\begin{equation}\label{eq20}
e_\Omega
:= \|z^*(T)-z_d\|_{(L^2(\Omega))^2}.
\end{equation}
For any measurable $\omega_\varepsilon\subseteq\Omega$, define
\begin{equation}\label{eq21}
\tilde{u}_{\omega_\varepsilon}(x,t)
:= \chi_{\omega_\varepsilon}(x)\,u^*_\Omega(x,t),
\end{equation}
which we call the truncated control and which equals $u^*_\Omega$ on $\omega_\varepsilon$ and zero on $\Omega\setminus\omega_\varepsilon$.

Since $0\leq\tilde{u}_{\omega_\varepsilon,i}\leq u_{b,i}$, we have $\tilde{u}_{\omega_\varepsilon}\in U_{\mathrm{ad}}$. We denote by
\begin{equation}\label{eq22}
z_{\omega_\varepsilon} := S_\Omega(\tilde{u}_{\omega_\varepsilon}),
\qquad \zeta := z_{\omega_\varepsilon} - z^*,
\end{equation}
the state under $\tilde{u}_{\omega_\varepsilon}$ and its difference with respect to $z^*$, and by
\begin{equation}\label{eq23}
r(\omega_\varepsilon)
:= \|\chi_{\Omega\setminus\omega_\varepsilon}\,u^*_\Omega
  \|_{(L^2(Q))^2}
\end{equation}
the residual.
Since $\chi_{\omega_\varepsilon}\tilde{u}_{\omega_\varepsilon}=\tilde{u}_{\omega_\varepsilon}$, the identity gives
\begin{equation}\label{eq24}
J(\tilde{u}_{\omega_\varepsilon};\omega_\varepsilon)
= J(\tilde{u}_{\omega_\varepsilon};\Omega).
\end{equation}
Since $\tilde{u}_{\omega_\varepsilon}\in U_{\mathrm{ad}}$, we have $J_{\omega_\varepsilon}^\star\leq J(\tilde{u}_{\omega_\varepsilon};\omega_\varepsilon)$ by~\eqref{eq9}, and~\eqref{eq24} gives
\begin{equation}\label{eq25}
J_{\omega_\varepsilon}^\star
\leq J(\tilde{u}_{\omega_\varepsilon};\Omega).
\end{equation}
It remains to estimate $J(\tilde{u}_{\omega_\varepsilon};\Omega)-J(u^*_\Omega;\Omega)$ in terms of $r(\omega_\varepsilon)$.
We subtract the state equations for $z_{\omega_\varepsilon}$ and $z^*$. The difference $\zeta=z_{\omega_\varepsilon}-z^*$ satisfies, for $i\in\{n,l\}$,
\begin{equation}\label{eq26}
\left\{
\begin{aligned}
\partial_t \zeta_i &= D_i\Delta \zeta_i
  - \gamma_i \zeta_i + \alpha_i[H(z_{\omega_\varepsilon})-H(z^*)]\\
  &\quad - \beta_i\chi_{\Omega\setminus\omega_\varepsilon}\,
    u^*_{i,\Omega},\\
\zeta_i(\cdot,0) &= 0,\quad
  \partial_\nu \zeta_i\big|_\Sigma = 0.
\end{aligned}
\right.
\end{equation}
Testing~\eqref{eq26} with the adjoint $p^*$ of Proposition~\ref{Prop2}(ii) gives the following result.

\begin{proposition}\label{Prop3}
Let Assumption~\ref{As1} hold, let $\omega_\varepsilon\subseteq\Omega$ be measurable, and let $u^*_\Omega\in U_{\mathrm{ad}}$ and $p^*$ satisfy the optimality conditions of Definition~\ref{RemLoc}, with $z^*:=S_\Omega(u^*_\Omega)$ the corresponding state of~\eqref{eq1}. Let $e_\Omega$ be the terminal error~\eqref{eq20}, let $z_{\omega_\varepsilon}$ and $\zeta$ be given by~\eqref{eq22} for the truncated control~\eqref{eq21}, and let $r(\omega_\varepsilon)$ be the residual~\eqref{eq23}. We define
\begin{align}\label{eq27}
E(\omega_\varepsilon)
&:= \sum_i\int_0^T\!\!\int_{\Omega
  \setminus\omega_\varepsilon}
  \beta_i\,u^*_{i,\Omega}\,
  (-p^*_i)\,dx\,dt,
\end{align}
\begin{align}\label{eq28}
\mathcal{R}(\omega_\varepsilon)
&:= \sum_i\int_0^T\!\!\int_\Omega
  \alpha_i\,p^*_i\bigl[
  H(z_{\omega_\varepsilon})-H(z^*)
\notag\\
&\qquad\qquad
  -\nabla_z H(z^*)\cdot \zeta
  \bigr]\,dx\,dt,
\end{align}
where $\nabla_zH$ is the gradient of $H$. Then
\begin{align}\label{eq29}
\frac{\mu}{2}\bigl(
  &\|z_{\omega_\varepsilon}(T)-z_d\|^2_{(L^2(\Omega))^2}
  - e_\Omega^2\bigr)
\notag\\
&= E(\omega_\varepsilon) + \mathcal{R}(\omega_\varepsilon)
  + \frac{\mu}{2}\|\zeta(T)\|^2_{(L^2(\Omega))^2}.
\end{align}
Moreover, $E(\omega_\varepsilon)\geq 0$, and
\begin{align}
\label{eq30}
E(\omega_\varepsilon) &\leq C_E\,r(\omega_\varepsilon)^2,\\
\label{eq31}
\|\zeta(T)\|^2_{(L^2(\Omega))^2}
&\leq L_S^2\,r(\omega_\varepsilon)^2,\\
\label{eq32}
|\mathcal{R}(\omega_\varepsilon)| &\leq C_{\mathcal{R}}\,r(\omega_\varepsilon)^2,
\end{align}
where $L_S$ is given by~\eqref{eq15} with any $\sigma>0$, $K_m$ is given in Lemma~\ref{Lem1}, $\|p^*_i\|_{L^\infty(Q)}<\infty$ by Proposition~\ref{Prop2}(ii), and
\begin{align}
\label{eq33}
C_E &:= \max_i\max\Big\{\lambda,\,
\frac{\beta_i\|p^*_i\|_{L^\infty(Q)}}{u_{b,i}}
\Big\},\\
\label{eq34}
C_{\mathcal{R}} &:= \frac{1}{2}\|D^2H\|_{L^\infty(K_m)}\,
T\,L_S^2\sum_i\alpha_i\|p^*_i\|_{L^\infty(Q)}.
\end{align}
\end{proposition}

The proof is given in Appendix~\ref{app2}.

\begin{remark}\label{RemTerm}
Let the assumptions of Proposition~\ref{Prop3} hold. By~\eqref{eq56} and the terminal condition in~\eqref{eq17}, $|E(\omega_\varepsilon)+\mathcal{R}(\omega_\varepsilon)|\leq\mu\,e_\Omega\,\|\zeta(T)\|_{(L^2(\Omega))^2}$, so the left-hand side of~\eqref{eq29} is close to $\frac{\mu}{2}\|\zeta(T)\|^2_{(L^2(\Omega))^2}$ when $e_\Omega$ is small compared with $\|\zeta(T)\|_{(L^2(\Omega))^2}$.
\end{remark}

Combining Proposition~\ref{Prop3} with~\eqref{eq25}, we obtain the main result.

\begin{theorem}\label{Th1}
Let Assumption~\ref{As1} hold, let $\omega_\varepsilon\subseteq\Omega$ be measurable, and let $u^*_\Omega\in U_{\mathrm{ad}}$ be a minimizer of the full-domain problem~\eqref{eq6}, with adjoint $p^*$ given by Proposition~\ref{Prop2}(ii), and let $\tilde{u}_{\omega_\varepsilon}$ and $r(\omega_\varepsilon)$ be the truncated control~\eqref{eq21} and the residual~\eqref{eq23}. Define
\begin{equation}\label{eq35}
C_r := C_E + C_{\mathcal{R}} + \frac{\mu}{2}L_S^2,
\end{equation}
where $L_S$ is given by~\eqref{eq15} with any $\sigma>0$, and $C_E$ and $C_{\mathcal{R}}$ are given by~\eqref{eq33} and~\eqref{eq34}. Then $C_r>\lambda/2$, the optimality loss~\eqref{eq10} and the residual~\eqref{eq23} satisfy
\begin{equation}\label{eq36}
\delta(\omega_\varepsilon)
\leq\Big(C_r-\frac{\lambda}{2}\Big)
r(\omega_\varepsilon)^2,
\end{equation}
and, for every $\eta>0$, $\delta(\omega_\varepsilon)\leq\eta$ whenever
\begin{equation}\label{eqcov}
|\omega_\varepsilon|\geq L^2-\frac{\eta}
{\big(C_r-\frac{\lambda}{2}\big)T\,U_b^2},
\end{equation}
where $U_b:=(u_{b,n}^2+u_{b,l}^2)^{1/2}$, and $L$ is the side length of $\Omega$, so that $|\Omega|=L^2$.
\end{theorem}

\begin{proof}
Since $u^*_\Omega$ minimizes~\eqref{eq6}, $J(u^*_\Omega;\Omega)=J_\Omega^\star$. The optimality loss~\eqref{eq10} is therefore $\delta(\omega_\varepsilon)=J_{\omega_\varepsilon}^\star-J(u^*_\Omega;\Omega)$, and~\eqref{eq25} bounds $J_{\omega_\varepsilon}^\star$ by $J(\tilde{u}_{\omega_\varepsilon};\Omega)$. Hence $\delta(\omega_\varepsilon)\leq J(\tilde{u}_{\omega_\varepsilon};\Omega)-J(u^*_\Omega;\Omega)$, the increase of the cost on $\Omega$ when $u^*_\Omega$ is replaced by the truncated control~\eqref{eq21}. For the control cost,~\eqref{eq21} and~\eqref{eq23} give
\begin{equation}\label{eq37}
\|\tilde{u}_{\omega_\varepsilon}\|^2_{(L^2(Q))^2}
= \|u^*_\Omega\|^2_{(L^2(Q))^2}
  - r(\omega_\varepsilon)^2,
\end{equation}
so the control cost in~\eqref{eq4} is smaller for $\tilde{u}_{\omega_\varepsilon}$ than for $u^*_\Omega$ by $\frac{\lambda}{2}r(\omega_\varepsilon)^2$. The terminal cost difference is~\eqref{eq29}. Adding the two differences gives
\begin{align}\label{eq38}
J(\tilde{u}_{\omega_\varepsilon};\Omega)-J(u^*_\Omega;\Omega)
&=E(\omega_\varepsilon)+\mathcal{R}(\omega_\varepsilon)
\notag\\
&\quad+\frac{\mu}{2}\|\zeta(T)\|^2-\frac{\lambda}{2}r(\omega_\varepsilon)^2.
\end{align}
It remains to bound the right-hand side of~\eqref{eq38}. With~\eqref{eq30}--\eqref{eq32},
\begin{align}\label{eq39}
\delta(\omega_\varepsilon)
&\leq E(\omega_\varepsilon)+|\mathcal{R}(\omega_\varepsilon)|
  +\frac{\mu}{2}\|\zeta(T)\|^2
  -\frac{\lambda}{2}r(\omega_\varepsilon)^2
\notag\\
&\leq \Big(C_r
  -\frac{\lambda}{2}\Big)r(\omega_\varepsilon)^2,
\end{align}
which is~\eqref{eq36}, with $C_r\geq C_E\geq\lambda>\lambda/2$ by~\eqref{eq33} and~\eqref{eq35}. Finally, since $0\leq u^*_{i,\Omega}\leq u_{b,i}$ a.e.\ in $Q$, the residual~\eqref{eq23} satisfies $r(\omega_\varepsilon)^2\leq(u_{b,n}^2+u_{b,l}^2)\,T\,|\Omega\setminus\omega_\varepsilon|=U_b^2\,T(L^2-|\omega_\varepsilon|)$, and~\eqref{eqcov} makes the right-hand side of~\eqref{eq36} at most $\eta$, which concludes the proof.
\end{proof}

Condition~\eqref{eqcov} constrains only the control area, so it gives the same coverage for the single square~\eqref{eq7}, where $|\omega_\varepsilon|=4\varepsilon^2$, and for the union~\eqref{eq8}, where $|\omega_\varepsilon|=4\sum_k\varepsilon_k^2$, whatever the number and the placement of the squares. Both families grow with $\varepsilon$ on $I=(0,L/2]$: the one-square family is the square~\eqref{eq7} centered at $(L/2,L/2)$, and the $M=q^2$-square family consists of squares of half-side $\varepsilon/q$ centered in the cells of a $q\times q$ partition of~$\Omega$, which gives the nine squares of Section~\ref{sec5} for $q=3$. Hence Lemma~\ref{Lem2} applies, and Theorem~\ref{Th1} bounds $\varepsilon^\star_\eta$ by the smallest $\varepsilon$ satisfying~\eqref{eqcov}. Section~\ref{sec5} studies the effect of the placement on the optimized controls.

\section{Numerical Results}\label{sec5}
\begin{figure}[t]
\centering
\begin{subfigure}[b]{\columnwidth}
  \centering
  \includegraphics[width=0.85\columnwidth,trim=0 3 0 4,clip]{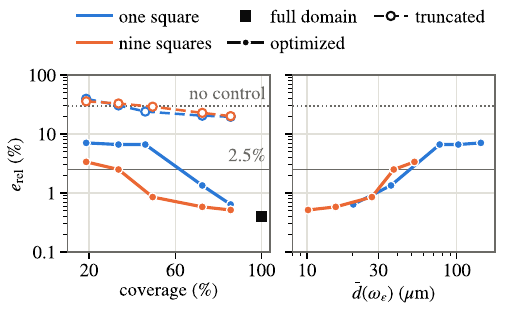}
  \caption{Target~1.}

\end{subfigure}\\[1pt]
\begin{subfigure}[b]{\columnwidth}
  \centering
  \includegraphics[width=0.85\columnwidth,trim=0 2.7 0 2.3,clip]{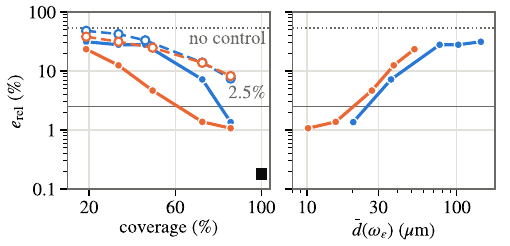}
  \caption{Target~2.}

\end{subfigure}
\caption{Relative terminal error $e_{\mathrm{rel}}$ in~\eqref{eq52} with one square (blue) and nine squares (orange), versus the coverage (left) and versus the mean distance $\bar d(\omega_\varepsilon)$ in~\eqref{eqdbar} for the optimized controls (right). Solid lines correspond to the control optimized on the subdomain, and dashed lines to the truncated control $\tilde{u}_{\omega_\varepsilon}$ in~\eqref{eq21}. The black squares mark the full-domain control, and the gray lines mark $e_{\mathrm{rel}}=2.5\%$ and the error of the uncontrolled state.}
\label{fig1}
\end{figure}
\begin{figure}[t]
\centering
\begin{subfigure}[b]{\columnwidth}
  \centering
  \includegraphics[width=0.85\columnwidth,trim=0 3 0 11.5,clip]{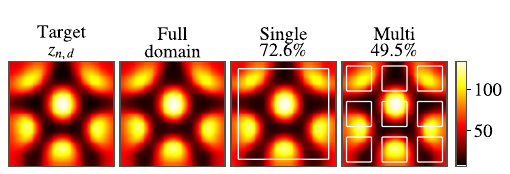}
  \caption{Target~1.}

\end{subfigure}\\[1pt]
\begin{subfigure}[b]{\columnwidth}
  \centering
  \includegraphics[width=0.85\columnwidth,trim=0 3 0 11.5,clip]{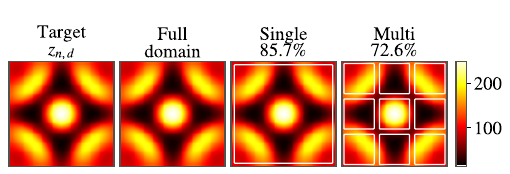}
  \caption{Target~2.}

\end{subfigure}
\caption{Nodal concentration $z_n(\cdot,T)$ at the final time. From left to right, the panels show the target~$z_{n,d}$ and the states under the full-domain control and under the controls optimized on one square (Single) and on nine squares (Multi), at the smallest simulated coverage with $e_{\mathrm{rel}}\leq 2.5\%$, and white lines delimit $\omega_\varepsilon$.}
\label{fig2}
\end{figure}
\begin{figure*}[t]
\centering
\includegraphics[width=0.89\textwidth,trim=0 9.5 0 12,clip]{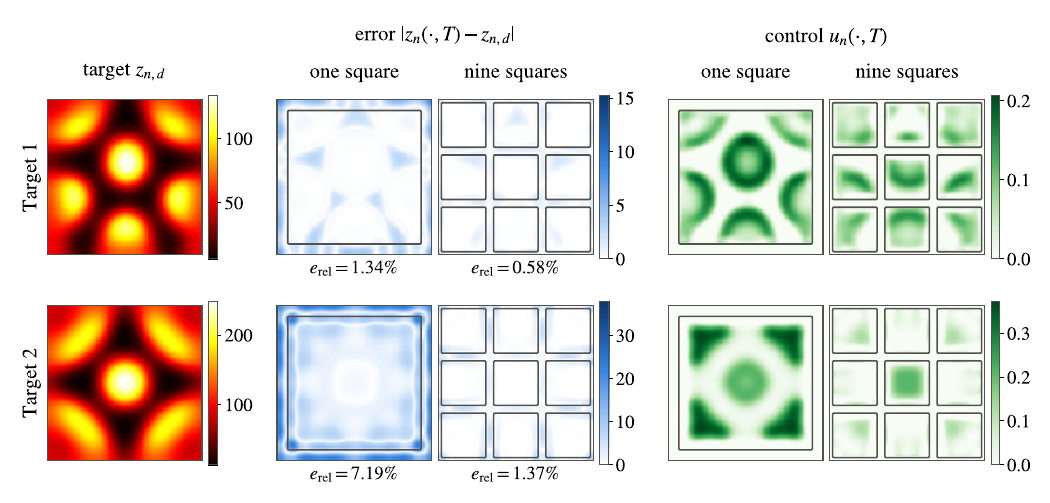}
\caption{Target $z_{n,d}$, terminal error $|z_n(\cdot,T)-z_{n,d}|$ and Nodal control $u_n(\cdot,T)$ at the final time, for one square and nine squares of coverage $72.6\%$. The top row corresponds to Target~1 and the bottom row to Target~2. Gray lines delimit $\omega_\varepsilon$, and the relative terminal error $e_{\mathrm{rel}}$ of each configuration is given below its error map.}
\label{figerr}
\end{figure*}

We study the Nodal--Lefty system~\eqref{eq1} numerically with the parameters listed in Table~\ref{tab1}. The side length of $\Omega$ is $L=800\;\mu$m, the control bounds are $u_{b,n}=u_{b,l}=1$, and the weights of the discretized cost are $\mu=1$ and $\lambda=10^{-8}$. The state equation is discretized on a uniform $81\times 81$ grid with the Crank--Nicolson scheme~\cite{CrankNicolson1947} for the diffusion terms and an explicit step for the reaction terms, and each optimal control problem is solved by the nonlinear conjugate gradient method with Polak--Ribi\`ere updates~\cite{PolakRibiere1969}, see~\cite{CDC_paper} for more details. The iterations stop when the cost stops decreasing or when a maximum number of iterations is reached.
The control problems use the parameters in Table~\ref{tab1} and the production rates $\alpha_n=0.8$ and $\alpha_l=4.0$\;nM\,min$^{-1}$, for which the uncontrolled system develops a stripe pattern from randomly perturbed initial data by Turing instability~\cite{Ouchdiri2025_MathBio}. We consider two targets, Target~1 (spotted) and Target~2 (radial), produced by~\eqref{eq1} with $(\alpha_n,\alpha_l)=(0.5,4.0)$ and $(1.5,8.0)$, over the horizons $T=2000$ and $5000$\;min; see~\cite{Ouchdiri2025_MathBio,CDC_paper} for details. The initial conditions are positive on the $81\times 81$ grid, with minimum above $0.01$~nM, the value of $m$ in Assumption~\ref{As1}(i).

\begin{table}[t]
\centering
\caption{Model parameters of the simulations. All values except $D_n$, $\beta_n$, and $\beta_l$ are from~\cite{Sekine2018}, which reports $D_n=1.96$. Units: $\mu$m$^2$min$^{-1}$ ($D_i$), min$^{-1}$ ($\gamma_i$), nM ($k_i$), nM\,min$^{-1}$ ($\beta_i$); $n_i$ dimensionless.}
\label{tab1}
\begin{tabular}{@{}llll@{}}
\toprule
$D_n$ & $1.94$ & $D_l$ & $56.39$\\
$\gamma_n$ & $2.37\!\times\!10^{-3}$
& $\gamma_l$ & $5.65\!\times\!10^{-3}$\\
$n_n$ & $2.63$ & $n_l$ & $1.09$\\
$k_n$ & $9.28$ & $k_l$ & $14.96$\\
$\beta_n$ & $0.8$ & $\beta_l$ & $4.0$\\
\bottomrule
\end{tabular}
\end{table}

For each target, we solve the full-domain problem~\eqref{eq6}, then the subdomain problem~\eqref{eq9} for five coverages, $18.7$, $33.7$, $46.1$ (one square) or $49.5$ (nine squares), $72.6$, and $85.7\%$. The single subdomain $\omega_\varepsilon$ is a square centered at $(L/2,L/2)$, and the multi-subdomain $\omega_\varepsilon$ consists of $M=9$ equal squares centered in the cells of a $3\times 3$ partition of~$\Omega$. Performance is measured by
\begin{equation}\label{eq52}
e_{\mathrm{rel}}
:= \frac{\|z(\cdot,T)-z_d\|_{(L^2(\Omega))^2}}
        {\|z_d\|_{(L^2(\Omega))^2}},
\end{equation}
where $z$ is the state under the control considered. In all optimized cases, the control cost is below $0.1\%$ of the total cost, so comparing costs amounts to comparing errors.
Figure~\ref{fig1} reports $e_{\mathrm{rel}}$ as a function of the coverage in its left panels. At each of the four coverages shared by the two configurations, the control optimized on nine squares gives a smaller error than the control optimized on one square, by a factor between $1.2$ and $5.2$. The reaction terms in~\eqref{eq1} act pointwise, so the control reaches the uncontrolled points, those of $\Omega\setminus\omega_\varepsilon$, only by diffusion out of $\omega_\varepsilon$. We measure how far these points lie from $\omega_\varepsilon$ by the mean distance
\begin{equation}\label{eqdbar}
\bar d(\omega_\varepsilon):=\frac{1}{|\Omega\setminus\omega_\varepsilon|}\int_{\Omega\setminus\omega_\varepsilon}\mathrm{dist}(x,\omega_\varepsilon)\,dx,
\end{equation}
which we evaluate on the $81\times 81$ grid, with $\mathrm{dist}$ measured to the nearest grid point of $\omega_\varepsilon$. The right panels of Figure~\ref{fig1} show $e_{\mathrm{rel}}$ as a function of $\bar d(\omega_\varepsilon)$. For each target, the error of the ten optimized controls increases with $\bar d(\omega_\varepsilon)$, except for one pair whose errors differ by less than $0.02$ percentage points. At these four coverages, the nine squares reduce $\bar d(\omega_\varepsilon)$ by a factor between $2.0$ and $2.8$. The optimized controls that reach $e_{\mathrm{rel}}\leq 2.5\%$ are those with $\bar d(\omega_\varepsilon)\leq 36.4\;\mu$m for Target~1 and $\bar d(\omega_\varepsilon)\leq 20.4\;\mu$m for Target~2, distances of the order of the Nodal decay length $\ell_n:=\sqrt{D_n/\gamma_n}\approx 28.6\;\mu$m, the distance Nodal diffuses out of $\omega_\varepsilon$ before it is degraded.
Figure~\ref{fig2} shows the targets and the final Nodal concentration, and Figure~\ref{figerr} shows the Nodal error and the Nodal control at the final time for the coverage $72.6\%$. At this coverage, $71\%$ to $82\%$ of the squared Nodal error lies in $\Omega\setminus\omega_\varepsilon$, which covers $27\%$ of $\Omega$, and nine squares reduce the mean of $|z_n(\cdot,T)-z_{n,d}|$ over $\Omega\setminus\omega_\varepsilon$ by factors of $3.1$ (Target~1) and $5.2$ (Target~2).

In Table~\ref{tab2}, the nine squares reach $e_{\mathrm{rel}}\leq 2.5\%$ with $23.1$ and $13.1$ percentage points less coverage than one square, for Target~1 and Target~2. For Target~1, the nine squares give $2.504\%$ at $33.7\%$, slightly above $2.5\%$. The truncated control gives $e_{\mathrm{rel}}>2.5\%$ at every simulated coverage, with errors of at least $19.6\%$ for Target~1 and $7.4\%$ for Target~2, while the full-domain control reaches $0.40\%$ and $0.18\%$. Since $\|\zeta(T)\|\geq\|z_{\omega_\varepsilon}(T)-z_d\|-e_\Omega$, Remark~\ref{RemTerm} shows that, for every truncated control in Figure~\ref{fig1}, the left-hand side of~\eqref{eq29} and $\frac{\mu}{2}\|\zeta(T)\|^2$ differ by at most $6\%$ of $\frac{\mu}{2}\|\zeta(T)\|^2$. The optimized errors are smaller than the truncated ones by factors between $1.2$ and $39$ in all twenty comparisons.

\begin{table}[t]
\centering
\caption{Smallest simulated coverage $(\%)$ with $e_{\mathrm{rel}}\leq 2.5\%$, for the controls optimized on $\omega_\varepsilon$ and for the truncated control~\eqref{eq21}.}
\label{tab2}
{\setlength{\tabcolsep}{8pt}
\renewcommand{\arraystretch}{1.15}
\begin{tabular}{@{}lcccc@{}}
\toprule
& \multicolumn{2}{c}{Target~1}
& \multicolumn{2}{c}{Target~2} \\
\cmidrule(lr){2-3}\cmidrule(lr){4-5}
& One & Nine & One & Nine \\
\midrule
Optimized & $72.6$ & $49.5$ & $85.7$ & $72.6$ \\
Truncated \eqref{eq21} & $>85.7$ & $>85.7$ & $>85.7$ & $>85.7$ \\
\bottomrule
\end{tabular}}
\end{table}

\section{Conclusion and Perspectives}\label{sec6}
We studied the optimality loss caused by restricting the control of a reaction--diffusion system to a subdomain. Setting an optimal control of the full-domain problem to zero outside the subdomain gives an admissible control for the subdomain problem, an identity for the difference between the costs of the two controls, and an upper bound on the optimality loss proportional to the squared space--time $L^2$ norm of this control outside the subdomain. This bound yields a coverage that keeps the optimality loss within a desired precision. For the Nodal--Lefty system with one parameter set and two targets, the simulations show that the controls optimized on nine squares give a smaller terminal error than those optimized on one square at the same coverage, and that the error falls below $2.5\%$ when the uncontrolled points lie, on average, within $36.4~\mu$m (Target~1) or $20.4~\mu$m (Target~2) of the control region. These distances are of the order of the Nodal decay length $\ell_n\approx 28.6~\mu$m, the distance Nodal diffuses out of the control region before it is degraded. Future work will focus on selecting the control region by minimizing the residual~\eqref{eq23} over regions of a given area. Beyond a region fixed for the whole horizon, we will study a control region that changes in time, as in the randomized sparse feedback framework of~\cite{Hadach2026}.
\appendix
\subsection{Proof of Proposition~\ref{Prop1}}\label{app1}
We need to bound the difference of the two states, which we do by a Gr\"onwall argument. Given two controls $u^1,u^2\in U_{\mathrm{ad}}$, let $z^1=S_\omega(u^1)$ and $z^2=S_\omega(u^2)$ be their corresponding states. Let $w_i:=z^1_i-z^2_i$ and $v_i:=u^1_i-u^2_i$ for $i\in\{n,l\}$. Subtracting the two state equations gives
\begin{align}\label{eq53}
\partial_t w_i &= D_i\Delta w_i - \gamma_i w_i \notag\\ &\quad + \alpha_i[H(z^1)-H(z^2)] + \beta_i\chi_\omega v_i
\end{align}
with $w_i(\cdot,0)=0$ and $\partial_\nu w_i|_\Sigma=0$. We multiply~\eqref{eq53} by $2w_i$, integrate over $\Omega$, and sum over $i$. The diffusion and degradation terms are nonpositive, the nonlinear terms satisfy $2\sum_i\alpha_i[H(z^1)-H(z^2)]w_i\leq 2\sqrt{2}\,\alpha_{\max}L_H|w|^2$ by~\eqref{eq3} and the Cauchy--Schwarz inequality in~$\R^2$, and $2\beta_i\chi_\omega v_i w_i\leq\sigma w_i^2+(\beta_{\max}^2/\sigma)v_i^2$ by Young's inequality. Hence
\begin{equation}\label{eq54}
\frac{d}{dt}\|w(t)\|^2 \leq \kappa_1\|w(t)\|^2 + \kappa_2\|v(t)\|^2
\end{equation}
with $\kappa_1:=2\sqrt{2}\,\alpha_{\max}L_H+\sigma$ and $\kappa_2:=\beta_{\max}^2/\sigma$. Since $w(\cdot,0)=0$, Gr\"onwall's inequality yields
\begin{equation}\label{eq55}
\sup_{t\in[0,T]}\|w(t)\|^2_{(L^2(\Omega))^2} \leq \kappa_2\,e^{\kappa_1 T}\, \|v\|^2_{(L^2(Q))^2},
\end{equation}
which is~\eqref{eq16} with $L_S$ as in~\eqref{eq15}. \hfill$\square$

\subsection{Proof of Proposition~\ref{Prop3}}\label{app2}
We prove the identity, then the sign of $E(\omega_\varepsilon)$, and then the three bounds. The identity~\eqref{eq29} follows from expanding $\|z_{\omega_\varepsilon}(T)-z_d\|^2 = \|\zeta(T)+(z^*(T)-z_d)\|^2$, using the terminal condition $p^*_i(T)=\mu(z^*_i(T)-z_{i,d})$, and the integration by parts identity
\begin{equation}\label{eq56}
\sum_i\!\int_\Omega \zeta_i(T)\,p^*_i(T)\,dx = E(\omega_\varepsilon) + \mathcal{R}(\omega_\varepsilon),
\end{equation}
obtained by multiplying~\eqref{eq26} by $p^*_i$, integrating over $Q$, summing over $i$, and integrating by parts in time and space using~\eqref{eq17} and the Neumann conditions for both $\zeta_i$ and $p^*_i$.

By~\eqref{eq19}, the integrand in~\eqref{eq27} is nonnegative, so $E(\omega_\varepsilon)\geq 0$. It remains to establish the three bounds. For~\eqref{eq30}, we use the projection formula~\eqref{eq18} pointwise on $(\Omega\setminus\omega_\varepsilon)\times(0,T)$, for each $i\in\{n,l\}$. Three cases arise.

If $u^*_{i,\Omega}=0$, the corresponding term in~\eqref{eq27} vanishes. If $0<u^*_{i,\Omega}<u_{b,i}$, the projection is in the interior of $[0,u_{b,i}]$, so $u^*_{i,\Omega}=-\beta_i p^*_i/\lambda$ and $\beta_i u^*_{i,\Omega}(-p^*_i)=\lambda(u^*_{i,\Omega})^2$. If $u^*_{i,\Omega}=u_{b,i}$, the bound $|p^*_i|\leq\|p^*_i\|_{L^\infty(Q)}$ gives $\beta_i u^*_{i,\Omega}(-p^*_i)\leq\beta_i\|p^*_i\|_{L^\infty(Q)}(u^*_{i,\Omega})^2/u_{b,i}$. In all three cases,
\begin{equation}\label{eq57}
\beta_i\, u^*_{i,\Omega}\,(-p^*_i) \leq C_{E,i}\,(u^*_{i,\Omega})^2
\end{equation}
with $C_{E,i}:=\max\{\lambda,\, \beta_i\|p^*_i\|_{L^\infty(Q)}/u_{b,i}\}$. Integrating over $(\Omega\setminus\omega_\varepsilon)\times(0,T)$ and summing over $i$ gives~\eqref{eq30} with $C_E=\max_i C_{E,i}$.

The bound~\eqref{eq31} follows directly from Proposition~\ref{Prop1} applied to $u^*_\Omega$ and $\tilde{u}_{\omega_\varepsilon}$. For~\eqref{eq32}, Lemma~\ref{Lem1} gives $z^*,z_{\omega_\varepsilon}\in K_m$ a.e.\ in $Q$, and $K_m$ is convex, so Taylor's theorem gives $\bigl|H(z_{\omega_\varepsilon})-H(z^*)-\nabla_z H(z^*)\cdot\zeta\bigr|\leq\tfrac{1}{2}\|D^2H\|_{L^\infty(K_m)}|\zeta|^2$ a.e.\ in $Q$. Inserting this into~\eqref{eq28} and using $|p^*_i|\leq\|p^*_i\|_{L^\infty(Q)}$ from Proposition~\ref{Prop2}(ii) gives $|\mathcal{R}(\omega_\varepsilon)|\leq\frac{1}{2}\|D^2H\|_{L^\infty(K_m)}\sum_i\alpha_i\|p^*_i\|_{L^\infty(Q)}\,\|\zeta\|^2_{(L^2(Q))^2}$. By Proposition~\ref{Prop1}, $\sup_{t\in[0,T]}\|\zeta(t)\|^2_{(L^2(\Omega))^2}\leq L_S^2\,r(\omega_\varepsilon)^2$, so $\|\zeta\|^2_{(L^2(Q))^2}\leq T\,L_S^2\,r(\omega_\varepsilon)^2$, which gives~\eqref{eq32}. \hfill$\square$

{\footnotesize
\begin{thebibliography}{99}\setlength{\itemsep}{-1.4pt}

\bibitem{Turing1952} A.\,M.~Turing,
``The chemical basis of morphogenesis,''
\emph{Phil.\ Trans.\ R.\ Soc.\ Lond.\ B},
vol.~237, no.~641, pp.~37--72, 1952.

\bibitem{Lions1971} J.-L.~Lions,
\emph{Optimal Control of Systems Governed
by Partial Differential Equations}.
Springer-Verlag, 1971.

\bibitem{Troeltzsch2010} F.~Tr\"{o}ltzsch,
\emph{Optimal Control of Partial
Differential Equations: Theory, Methods
and Applications}.
AMS, 2010.

\bibitem{LasieckaTriggiani2000} I.~Lasiecka and R.~Triggiani, \emph{Control Theory for Partial Differential Equations: Continuous and Approximation Theories}, vols.~1--2. Cambridge, U.K.: Cambridge Univ.\ Press, 2000.

\bibitem{CasasRyllTroltzsch2013} E.~Casas, C.~Ryll, and F.~Tr\"oltzsch, ``Sparse optimal control of the Schl\"ogl and FitzHugh--Nagumo systems,'' \emph{Comput.\ Methods Appl.\ Math.}, vol.~13, no.~4, pp.~415--442, 2013.

\bibitem{GarvieTrenchea2014} M.\,R.~Garvie and C.~Trenchea, ``Identification of space-time distributed parameters in the Gierer--Meinhardt reaction-diffusion system,'' \emph{SIAM J.\ Appl.\ Math.}, vol.~74, no.~1, pp.~147--166, 2014.

\bibitem{LebeauRobbiano1995} G.~Lebeau and L.~Robbiano, ``Contr\^ole exact de l'\'equation de la chaleur,'' \emph{Comm.\ Partial Differential Equations}, vol.~20, pp.~335--356, 1995.

\bibitem{Miller2004} L.~Miller, ``Geometric bounds on the growth rate of null-controllability cost for the heat equation in small time,'' \emph{J.\ Differential Equations}, vol.~204, no.~1, pp.~202--226, 2004.

\bibitem{Loheac2017} J.~Loh\'eac, E.~Tr\'elat, and E.~Zuazua, ``Minimal controllability time for the heat equation under unilateral state or control constraints,'' \emph{Math.\ Models Methods Appl.\ Sci.}, vol.~27, pp.~1587--1644, 2017.

\bibitem{KrsticSmyshlyaev2008} M.~Krstic and A.~Smyshlyaev, \emph{Boundary Control of PDEs: A Course on Backstepping Designs}. Philadelphia, PA: SIAM, 2008.

\bibitem{Kubrusly1985} C.\,S.~Kubrusly and H.~Malebranche, ``Sensors and controllers location in distributed systems---A survey,'' \emph{Automatica}, vol.~21, no.~2, pp.~117--128, 1985.

\bibitem{Morris2011} K.~Morris,
``Linear-quadratic optimal actuator
location,''
\emph{IEEE Trans.\ Automat.\ Control},
vol.~56, no.~1, pp.~113--124, 2011.

\bibitem{Edalatzadeh2019} M.\,S.~Edalatzadeh and K.\,A.~Morris, ``Optimal actuator design for semilinear systems,'' \emph{SIAM J.\ Control Optim.}, vol.~57, no.~4, pp.~2992--3020, 2019.

\bibitem{Buchholz2013} R.~Buchholz, H.~Engel, E.~Kammann, and F.~Tr\"oltzsch, ``On the optimal control of the Schl\"ogl-model,'' \emph{Comput.\ Optim.\ Appl.}, vol.~56, no.~1, pp.~153--185, 2013.

\bibitem{Stadler2009} G.~Stadler, ``Elliptic optimal control problems with $L^1$-control cost and applications for the placement of control devices,'' \emph{Comput.\ Optim.\ Appl.}, vol.~44, no.~2, pp.~159--181, 2009.

\bibitem{Herzog2012} R.~Herzog, G.~Stadler, and G.~Wachsmuth, ``Directional sparsity in optimal control of partial differential equations,'' \emph{SIAM J.\ Control Optim.}, vol.~50, no.~2, pp.~943--963, 2012.

\bibitem{Casas2017} E.~Casas, R.~Herzog, and G.~Wachsmuth, ``Analysis of spatio-temporally sparse optimal control problems of semilinear parabolic equations,'' \emph{ESAIM Control Optim.\ Calc.\ Var.}, vol.~23, no.~1, pp.~263--295, 2017.

\bibitem{Ouchdiri2025_MathBio} M.\,A.~Ouchdiri, S.~Benjelloun, A.~Saoud, and I.~Otero-Muras, ``Turing patterns in a morphogenetic model with single regulatory function,'' \emph{Math.\ Biosci.}, vol.~389, Art.~no.~109536, 2025.

\bibitem{Sekine2018} R.~Sekine, T.~Shibata, and M.~Ebisuya,
``Synthetic mammalian pattern formation
driven by differential diffusivity of
Nodal and Lefty,''
\emph{Nat.\ Commun.}, vol.~9, no.~1,
Art.~no.~5456, 2018.

\bibitem{McNamara2025} H.\,M.~McNamara, A.\,M.~Guyer, B.\,Z.~Jia, V.\,J.~Parot, C.\,D.~Dobbs, A.\,F.~Schier, A.\,E.~Cohen, and N.\,D.~Lord, ``Optogenetic control of Nodal signaling patterns,'' \emph{Development}, vol.~152, no.~9, Art.~no.~dev204506, 2025.

\bibitem{Krueger2019} D.~Krueger, E.~Izquierdo, R.~Viswanathan, J.~Hartmann, C.~Pallares~Cartes, and S.~De~Renzis, ``Principles and applications of optogenetics in developmental biology,'' \emph{Development}, vol.~146, no.~20, Art.~no.~dev175067, 2019.

\bibitem{CDC_paper} M.\,A.~Ouchdiri, H.~Faquir, S.~Benjelloun, M.~Maghenem, I.~Otero-Muras, and A.~Saoud, ``An optimal-control framework for reaction diffusion systems with application to synthetic developmental biology,'' in \emph{Proc.\ 64th IEEE Conf.\ Decision and Control}, 2025, pp.~1925--1930.

\bibitem{Ebrahimkhani2019} M.\,R.~Ebrahimkhani and M.~Ebisuya,
``Synthetic developmental biology: build
and control multicellular systems,''
\emph{Curr.\ Opin.\ Chem.\ Biol.}, vol.~52,
pp.~9--15, 2019.

\bibitem{Barbu2012} V.~Barbu,
\emph{Mathematical Methods in Optimization
of Differential Systems}.
Springer, 2012.

\bibitem{Pao1992} C.\,V.~Pao,
\emph{Nonlinear Parabolic and Elliptic
Equations}.
Plenum Press, New York, 1992.

\bibitem{Evans2010} L.\,C.~Evans,
\emph{Partial Differential Equations},
2nd ed., Graduate Studies in Mathematics,
vol.~19.
AMS, 2010.

\bibitem{CrankNicolson1947} J.~Crank and P.~Nicolson,
``A practical method for numerical
evaluation of solutions of partial
differential equations of the
heat-conduction type,''
\emph{Math.\ Proc.\ Camb.\ Phil.\ Soc.},
vol.~43, no.~1, pp.~50--67, 1947.

\bibitem{PolakRibiere1969} E.~Polak and G.~Ribi\`ere,
``Note sur la convergence de m\'ethodes
de directions conjugu\'ees,''
\emph{Rev.\ Fr.\ Inform.\ Rech.\
Op\'er.}, vol.~3, no.~16, pp.~35--43,
1969.

\bibitem{Hadach2026} Z.~Hadach, H.~El~Hammouti,
E.\,H.~Bergou, and A.~Saoud,
``Just few states are enough:
randomized sparse feedback for
stability of dynamical systems,''
in \emph{Proc.\ AAAI Conf.\ Artificial
Intelligence}, pp.~18271--18278, 2026.

\end{thebibliography}}
\end{document}